\documentclass[11pt,draft]{amsart}
\usepackage{mathrsfs}
\usepackage{amssymb}
\usepackage{amsmath}
\usepackage{amsfonts}
\usepackage{amsfonts,enumerate}
\usepackage{pifont}
\usepackage{enumerate}
\usepackage{graphics}
\usepackage{verbatim}
\usepackage{amsthm}

\numberwithin{equation}{section}
\newtheorem{theorem}{Theorem}[section]
\newtheorem{corollary}{Corollary}[section]
\newtheorem{definition}{Definition}[section]
\newtheorem{lemma}{Lemma}[section]

\newtheorem{remark}{Remark}[section]
\newtheorem{example}{Example}[section]

\newcommand{\M}{{\mathcal M}}
\newcommand{\N}{{\mathcal N}}

\newcommand{\8}{\infty}
\newcommand{\el}{\ell}

\newcommand{\be}{\begin{eqnarray*}}
\newcommand{\ee}{\end{eqnarray*}}
\newcommand{\beq}{\begin{equation}}
\newcommand{\eeq}{\end{equation}}
\newcommand{\beqn}{\begin{equation*}}
\newcommand{\eeqn}{\end{equation*}}
\newcommand{\bs}{\begin{split}}
\newcommand{\es}{\end{split}}

\numberwithin{equation}{section}

\begin{document}

\title{Interpolation and $\Phi$-moment inequalities of noncommutative martingales}


\author{Turdebek N. Bekjan}

\address{College of Mathematics and Systems Sciences, Xinjiang
University, \newblock Urumqi 830046, China}

\thanks{T.B is partially supported by NSFC grant No.10761009}

\author{Zeqian Chen}

\address{Wuhan Institute of Physics and Mathematics, Chinese
Academy of Sciences, 30 West District, Xiao-Hong-Shan, Wuhan 430071, China}

\thanks{Z.C is partially supported by NSFC grant No.10775175.}

\date{}
\maketitle

\markboth{T.N.Bekjan and Z. Chen}%
{$\Phi$-moment inequalities of noncommutative martingales}

\begin{abstract}
This paper is devoted to the study of $\Phi$-moment inequalities for noncommutative martingales. In particular, we prove the noncommutative $\Phi$-moment analogues of martingale transformations, Stein's inequalities, Khintchine's inequalities for Rademacher's random variables, and Burkholder-Gundy's inequalities. The key ingredient is a noncommutative version of Marcinkiewicz type interpolation theorem for Orlicz spaces which we establish in this paper.
\end{abstract}


\section{Introduction}\label{intro}

Given a probability space $(\Omega,\mathscr{F},P).$ Let $\{\mathscr{F}_{n}\}_{n\ge1}$ be a nondecreasing sequence of $\sigma$-subfields of $\mathscr{F}$ such that $\mathscr{F}=\vee\mathscr{F}_{n}$, and let $\Phi$ be an Orlicz function with $1<p_{\Phi}\leq q_{\Phi}<\infty.$ If $f=(f_{n})_{n\ge1}$ is a $L_{\Phi}$-bounded martingale, then
\begin{equation}\label{eq:PhiBG}
\int_{\Omega}\Phi \Big [ \Big ( \sum_{n = 1}^{\infty}|d f_n |^{2} \Big )^{\frac{1}{2}} \Big ] d P \approx
\sup_{n\ge1}\int_{\Omega}\Phi(|f_{n}|)dP.
\end{equation}
where $d f = ( d f_n )_{n \geq 1}$ is the martingale difference of $f$ and $``\approx"$ depends only on $\Phi.$ This result is the well-known Burkholder-Gundy inequality for convex powers $\Phi (t) = t^p$ (see \cite{BG}) and proved in the general setting of Orlicz functions by Burkholder-Davis-Gundy \cite{BDG}. In their remarkable paper \cite{PX1}, Pisier and Xu proved the noncommutative analogue of the Burkholder-Gundy inequality, which triggered a systematic research of noncommutative martingale inequalities (we refer to a recent book by Xu \cite{X2} for an up-to-date exposition of theory of noncommutative martingales). In this paper, we will extend their work to $\Phi$-moment versions, i.e., we will prove the noncommutative analogue of \eqref{eq:PhiBG}.

Let us briefly describe our $\Phi$-moment inequality. Let $\M$ be a finite von Neumann algebra with a normalized normal faithful trace $\tau,$ and $\{\M_n\}_{n\ge 0}$ be an increasing filtration of von Neumann subalgebras of $\M.$ Let $\Phi$ be an Orlicz function and $x=\{x_{n}\}_{n\geq 0}$ a noncommutative $L_{\Phi}$-martingale with respect to $\{\M_n\}_{n \ge 0}.$ Then our result reads as follows. If $1<p_{\Phi} \le q_{\Phi}<2,$ then
\begin {equation}\label{eq:PhiBG1}
\begin{split}
\tau \big ( \Phi [ |x|] \big ) \approx \inf \Big \{ \tau \Big ( \Phi \Big [ \Big ( \sum_{k= 0 }^{\infty} |dy_{k}|^{2} \Big )^{ \frac{1}{2}} \Big ] \Big ) + \tau \Big ( \Phi \Big [ \Big ( \sum_{k= 0 }^{\infty} |dz_{k}^{*}|^{2} \Big )^{ \frac{1}{2}} \Big ] \Big ) \Big \}
\end{split}
\end{equation}
where the infimum runs over all decomposition $x_{k}=y_{k}+z_{k}$
with $y_{k}$ in $\mathcal{H}_{C}^{\Phi}({\mathcal{M}})$ and $z_{k}$ in
$\mathcal{H}_{R}^{\Phi}({\mathcal{M}}),$ and if $2 <p_{\Phi} \le q_{\Phi}<\infty,$ then
\begin {equation}\label{eq:PhiBG2}
\tau \big ( \Phi [ |x|] \big ) \approx \max \Big \{ \tau \Big ( \Phi \Big [ \Big ( \sum_{k= 0 }^{\infty} |dx_{k}|^{2} \Big )^{ \frac{1}{2}} \Big ] \Big ),\; \tau \Big ( \Phi \Big [ \Big ( \sum_{k= 0 }^{\infty} |dx_{k}^{*}|^{2} \Big )^{ \frac{1}{2}} \Big ] \Big ) \Big \}.
\end{equation}
Here ``$\approx$" depends only on $\Phi.$ Note that the Orlicz norm version of noncommutative analogue of \eqref{eq:PhiBG} has been proved by the first named author \cite{B}. Evidently, the $\Phi$-moment inequalities imply the norm version.

One interesting feature of our result, similar to that of Pisier-Xu \cite{PX1}, is that the square function is defined differently (and it must be changed!) according to $q_{\Phi} <2$ or $p_{\Phi}>2.$ This surprising phenomenon was already discovered by F.Lust-Piquard in \cite{LP1, LP2} (see also \cite{LPP}) while establishing noncommutative versions of Khintchine's inequalities (see $\S 5$ also).

Stopping times and good-$\lambda$ techniques developed by Burkholder {\it etal} \cite{Burk} are two key ingredients in the proof of \eqref{eq:PhiBG}. Unfortunately, the concept of stopping times is, up to now, not well defined in the generic noncommutative setting (there are some works on this topic, see \cite{AC} and references therein). On the other hand, the noncommutative analogue of good-$\lambda$ inequalities seems open. Then, in order to prove the noncommutative $\Phi$-moment inequalities \eqref{eq:PhiBG1} and \eqref{eq:PhiBG2} we need new ideas.

The style of proof of \eqref{eq:PhiBG1} and \eqref{eq:PhiBG2} is via interpolation. Our key ingredient is a noncommutative analogue of Marcinkiewicz type interpolation theorem for Orlicz spaces, which we will prove in this paper. Recall that the first interpolation theorem concerning Orlicz spaces as intermediate spaces is due to Orlicz \cite{Or}. Subsequently, the classical Marcinkiewicz interpolation theorem was extended to include Orlicz spaces as interpolation classes by A.Zygmund, A.P.Calder\'{o}n, {\it et al.}, for references see \cite{M2} and therein.

Now, let us briefly explain our strategy. Firstly, we prove $\Phi$-moment versions of noncommutative martingale transforms and Stein's inequalities via interpolation. Then by interpolation again we prove $\Phi$-moment versions of noncommutative Khintchine's inequalities (this is the key point of the proof). Finally, by these $\Phi$-moment inequalities we deduce \eqref{eq:PhiBG1} and \eqref{eq:PhiBG2}. This argument seems new and that even in the classical case, it is simpler than all existing methods to the $\Phi$-moment inequalities of martingales.

The remainder of this paper is divided into six sections. In Section \ref{pre}, we present some preliminaries and notations on the noncommutative Orlicz spaces and Orlicz-Hardy spaces of noncommutative martingales. Then, a noncommutative analogue of Marcinkiewicz type interpolation theorem for Orlicz spaces is proved in Section \ref{inter}, which is the key ingredient for the proof of the main result in this paper. $\Phi$-moment versions of noncommutative martingale transforms and Stein's inequalities are proved in Section \ref{martran}. As an immediate application of $\Phi$-moment inequalities of noncommutative martingale transforms, we will prove the UMD property of noncommutative Orlicz spaces. In Section \ref{khint}, the noncommutative $\Phi$-moment Khintchine inequalities for Rademacher's random variables are proved via interpolation again. By the $\Phi$-moment inequalities proved previously, we deduce the $\Phi$-moment version of noncommutative Burkholder-Gundy's martingale inequalities in Section \ref{BG}. Finally, in Section \ref{re}, we make some remarks on our results and possible further researches.

In what follows, $C$ always denotes a constant, which may be
different in different places. For two nonnegative (possibly
infinite) quantities $X$ and $Y$ by $X \approx Y$ we mean that there
exists a constant $C>1$ such that $C^{-1} X \leq Y \leq C X.$

\section{Preliminaries}\label{pre}

\subsection{Noncommutative Orlicz spaces}

We use standard notation and notions from theory of noncommutative
$L_{p}$-spaces. Our main references are \cite{PX2} and \cite{X2} (see
also \cite{PX2} for more  historical references). Let  $\mathcal{N}$
be a semifinite von Neumann algebra acting on a Hilbert space
$\mathbb{H}$ with a normal semifinite faithful trace $\nu.$ Let
$L_{0}(\mathcal{N})$ denote the topological $*$-algebra of
measurable operators with respect to $(\mathcal{N}, \nu)$. The
topology of $L_{0}(\mathcal{N})$ is determined by the convergence of
measure. The trace $\nu$ can be extended to the positive cone
$L_{0}^{+}(\mathcal{N})$ of $L_{0}(\mathcal{N}):$
$$
\nu(x)= \int_{0}^{\infty}\lambda\,d\nu(E_{\lambda}(x)),
$$
where  $x=\int_{0}^{\infty}\lambda\,dE_{\lambda}(x)$ is the spectral
decomposition of $x$. Given $0<p<\infty$, let
$$
L_{p}(\mathcal{N})=\{x\in L_{0}(\mathcal{N}):\;
 \nu(|x|^{p})^{\frac{1}{p}}<\infty\}.
$$
We define
$$
\|x\|_{p}= \nu(|x|^{p})^{\frac{1}{p}},\quad x\in
L_{p}(\mathcal{N}).
$$
Then, $(L_{p}(\mathcal{N}),\|.\|_{p})$ is a Banach (or quasi-Banach for
$p<1$) space. They are the noncommutative $L_{p}$-space associated
with $(\mathcal{N},\nu)$, denoted by $L_{p}(\mathcal{N},\nu)$
or simply by $L_{p}(\mathcal{N})$. As usual, we set
$L_{\infty}(\mathcal{N},\nu)=\mathcal{N}$ equipped with the
operator norm.

\begin{definition} Let $\mathcal{N}$ be a semifinite von Neumann algebra on a Hilbert
space $\mathbb{H}$ with a normal semifinite faithful  trace
$\nu$. Let $x\in L_{0}(\mathcal{N})$. Define
$$
\lambda_{s}(x)=\nu(E_{(s,\infty)}(|x|)),\quad s>0,
$$
where $E_{(s,\infty)}(|x|)$ is the spectral projection of $x$
associated with the interval $(s,\infty))$. The function
$s\mapsto\lambda_{s}(x)$ is called the distribution function of $x$.
\end{definition}

For $0<p<\infty,$ we have the following Kolmogorov inequality:
\begin{equation}\label{eq:kol}
\lambda_{s}(x)\leq\frac{\|x\|_{p}^{p}}{s^{p}},\quad \forall x\in
L_{p}(\mathcal{N}).
\end{equation}

\begin{definition} Let $x$ be a $\tau$-measure operator and
$t>0.$ The ``$t$-th singular number of $x$" $\mu_{t}(x)$ is defined by
$$
\mu_{t}(x)=\inf\left\{\|xe\|: \;e\;\mbox{is any projection
in}\:{\mathcal N}\;\mbox{with}\;\tau(e^{\perp})\leq t\right\}.
$$
\end{definition}

The $\mu_{.}(x)$  is finite valued and decreasing function on
$(0,\infty).$ For further information on the generalised singular
value we refer the reader to \cite{FK}.

Let $\Phi$  be an Orlicz function on $[0,\infty),$ i.e., a continuous increasing and convex function satisfying $\Phi(0)=0$ and $\lim_{t\rightarrow
\infty}\Phi(t)=\infty.$ Recall that $\Phi$ is said to satisfy the
$\triangle_2$-condition if there is a constant $C$ such that $\Phi(2t)\leq C\Phi(t)$ for all $t>0.$ In this case, we denote by $\Phi \in \Delta_2.$
It is easy to check that $\Phi \in \triangle_2$ if and only if for any $a > 0$ there is a constant $C_a>0$ such that $\Phi(a t)\leq C_a \Phi(t)$ for all $t>0.$

We will work with some standard indices associated to an Orlicz function. Given an Orlicz function $\Phi.$ Let
\be M(t, \Phi)= \sup_{s >0} \frac{\Phi (t s)}{\Phi (s)},\quad t >0.
\ee
Define
\be
p_{\Phi} = \lim_{t \to 0} \frac{\log M(t, \Phi)}{\log t}, \quad q_{\Phi} = \lim_{t \to \8} \frac{\log M(t, \Phi)}{\log t}.
\ee
All the following properties we will use in the sequel are classical and can be found in \cite{M1}:
\begin{enumerate}[{\rm (1)}]

\item $1 \le p_{\Phi} \le q_{\Phi} \le \8.$

\item We have the following characterizations of $p_{\Phi}$ and $q_{\Phi}:$
\be p_{\Phi} = \sup \Big \{ p >0:\; \int^t_0 s^{-p} \Phi (s) \frac{d s}{s} = O(t^{- p} \Phi (t)),\; \forall t >0 \Big \};\ee
\be q_{\Phi} = \inf \Big \{ q >0:\; \int^{\8}_t s^{-q} \Phi (s) \frac{d s}{s} = O(t^{- q} \Phi (t)),\; \forall t >0 \Big \}.\ee

\item $\Phi \in \triangle_2$ if and only if $q_{\Phi} < \8$ if and only if $ \sup_{t>0} t \Phi'(t)/\Phi(t)< \8.$ ($\Phi' (t)$ is defined for each $t > 0$ except for a countable set of points in which we take $\Phi'(t)$ as the derivative from the right.)

\end{enumerate}
See \cite{M1, M2} for more information on Orlicz functions and Orlicz spaces.

For an Orlicz function $\Phi,$ the noncommutative Orlicz space
$L_{\Phi}(\mathcal{N})$ is defined as the space of all measurable operators with
respect to $(\mathcal{N},\nu)$ such that
$$\nu \Big ( \Phi \Big ( \frac{|x|}{c} \Big ) \Big )<\infty$$
for some $c>0.$ The space $L_{\Phi}(\mathcal{N}),$ equipped with the norm
\be
\|x\|_{\Phi}= \inf \big \{c>0: \;\nu \big ( \Phi({|x|}/{c}) \big )<1 \big \},
\ee
is a Banach space. If $\Phi(t)=t^p$ with $1 \leq p<\infty$ then
$L_\Phi(\mathcal{N})= L_p(\mathcal{N}).$ Noncommutative Orlicz spaces are symmetric spaces of measurable operators as defined in \cite{X1}.

Let $a=(a_{n})$ be a finite sequence in $L_{\Phi}({\mathcal{N}}),$
we define
\be
\|a\|_{L_{\Phi}({\mathcal{N}},\el_{C}^{2})}= \Big \| \Big ( \sum_n
|a_{n}|^{2} \Big )^{1/2} \Big \|_{\Phi}\;
\text{and}\;
\|a\|_{L_{\Phi}({\mathcal{N}},\el_{R}^{2})} = \Big \| \Big ( \sum_{n\geq 0
} |a_{n}^{*}|^{2} \Big )^{1/2} \Big \|_{\Phi},
\ee
respectively. This gives two norms on the family of all finite sequences in
$L_{\Phi}({\mathcal{N}}).$ To see this, let us consider the von Neumann
algebra tensor product ${\mathcal{N}}\otimes{ \mathcal{B}}(\el^{2})$
with the product trace $\nu \otimes \mathrm{tr},$ where
${\mathcal{B}}(\el^{2}) $ is the algebra of all bounded operators on
$\el^{2}$ with the usual trace $\mathrm{tr}.$ $\nu \otimes \mathrm{tr}$ is a
semifinite normal faithful trace. The associated noncommutative
Orlicz space is denoted by $L_{\Phi}({\mathcal{N}}\otimes
{\mathcal{B}}(\el^{2})).$ Now, any finite sequence $a=(a_{n})_{n\geq
0}$ in $L_{\Phi}({\mathcal{N}})$ can be regarded as an element in
$L_{\Phi}({\mathcal{N}}\otimes {\mathcal{B}}(\el^{2}))$ via the
following map
\be
a\longmapsto T(a)= \left(
\begin{matrix}
a_{0} & 0 & \ldots \\
a_{1} & 0 & \ldots \\
\vdots & \vdots & \ddots
\end {matrix}
\right ),
\ee
that is, the matrix of $T(a)$ has all vanishing entries except those
in the first column which are the ${a_{n}}$'s. Such a matrix is
called a column matrix, and the closure in
$L_{\Phi}({\mathcal{N}}\otimes {\mathcal{B}}(\el^{2}))$ of all column
matrices is called the column subspace of
$L_{\Phi}({\mathcal{N}}\otimes {\mathcal{B}}(\el^{2})).$ Since
\be
\|a\|_{L_{\Phi}({\mathcal{N}},\el_{C}^{2})}=\||T(a)|\|_{L_{\Phi}({\mathcal{N}}\otimes
{\mathcal{B}}(\el^{2}))}= \|T(a)\|_{L_{\Phi}({\mathcal{N}}\otimes
{\mathcal{B}}(\el^{2}))},
\ee
then $\|.\|_{L_{\Phi}({\mathcal{N}},\el_{C}^{2})}$ defines a norm
on the family of all finite sequences of $L_{\Phi}({\mathcal{N}}).$
The corresponding completion $L_{\Phi}({\mathcal{N}},\el_{C}^{2})$ is a Banach space. It is clear that a sequence
$a=(a_{n})_{n\geq 0}$ in $L_{\Phi}({\mathcal{N}})$ belongs to
$L_{\Phi}({\mathcal{N}}, \el_{C}^{2})$ if and only if
\be
\sup_{n\geq 0} \Big \| \Big ( \sum_{k= 0
}^{n}|a_{k}|^{2} \Big )^{1/2} \Big \|_{\Phi}<\infty.
\ee
If this is the case, $\big ( \sum_{k= 0
}^{\infty}|a_{k}|^{2} \big )^{1/2}$ can be appropriately defined as an
element of $L_{\Phi}({\mathcal{N}}).$ Similarly, $
\|.\|_{L_{\Phi}({\mathcal{N}},\el_{R}^{2})}$ is also a norm on the family
of all finite sequence in $L_{\Phi}({\mathcal{N}}),$ and the corresponding completion $L_{\Phi}({\mathcal{N}}, \el_{R}^{2})$ is a Banach space, which is isometric to the row subspace of
$L_{\Phi}({\mathcal{N}}\otimes {\mathcal{B}}(\el^{2}))$ consisting of
matrices whose nonzero entries lie only in the first row. Observe
that the column and row subspaces of $L_{\Phi}({\mathcal{N}}\otimes
{\mathcal{B}}(\el^{2}))$ are 1-complemented by Theorem 3.4
in \cite{DDP2}.

\begin{definition} Let $\Phi$ be an Orlicz function. The space
$CR_{\Phi}[L_{\Phi}({\mathcal{N}})]$ is defined as follows:
\begin{enumerate}[\rm (1)]

\item If $q_{\Phi} < 2,$
\be
CR_{\Phi}[L_{\Phi}({\mathcal{N}})]=L_{\Phi}({\mathcal{N}}, \el_{C}^{2})+
L_{\Phi}({\mathcal{N}}, \el_{R}^{2})
\ee
equipped with the sum norm:
\be
\|(x_{n})\|_{CR_{\Phi}[L_{\Phi}({\mathcal{N}})]}=\inf \big \{ \|(y_{n})\|_{L_{\Phi}({\mathcal{N}}, \el_{C}^{2})},\, \|(y_{n})\|_{L_{\Phi}({\mathcal{N}}, \el_{R}^{2})} \big \},
\ee
where the infimun runs over all decomposition $x_{n}=y_{n}+z_{n}$
with $y_{n}$ and $z_{n}$ in $L_{\Phi}({\mathcal{N}})$.

\item  If $2\leq p_{\Phi},$
\be
CR_{\Phi}[L_{\Phi}({\mathcal{N}})]=L_{\Phi}({\mathcal{N}}, \el_{C}^{2}) \cap
L_{\Phi}({\mathcal{N}}, \el_{R}^{2})
\ee
equipped with the intersection norm:
\be
\|(x_{n})\|_{CR_{\Phi}[L_{\Phi}({\mathcal{N}})]}=\max \big \{\|(x_{n})\|_{L_{\Phi}({\mathcal{N}}, \el_{C}^{2})},\, \|(x_{n})\|_{L_{\Phi}({\mathcal{N}}, \el_{R}^{2})} \big \}.
\ee
\end{enumerate}
\end{definition}


In the sequel, unless otherwise specified, we always denote by $\Phi$ an Orlicz function.

\subsection{Noncommutative martingales}

Let ${\mathcal{M}}$ be a finite von Neaumann algebra
with a normalized normal faithful trace $\tau.$ Let
$({\mathcal{M}}_{n})_{n\geq 0}$ be an increasing sequence of von
Neaumann subalgebras of ${\mathcal{M}}$ such that $\cup_{n\geq 0}
{\mathcal{M}}_{n}$ generates ${\mathcal{M}}$ (in the
$w^{*}$-topology). $({\mathcal{M}}_{n})_{n\geq 0}$ is called a
filtration of ${\mathcal{M}}.$ The restriction of $\tau$ to
${\mathcal{M}}_{n}$ is still denoted by $\tau.$ Let
${\mathcal{E}}_{n}={\mathcal{E}}(.|{\mathcal{M}}_{n})$ be the
conditional expectation of ${\mathcal{M}}$ with respect to
${\mathcal{M}}_{n}.$ Then ${\mathcal{E}}_{n}$ is a norm 1 projection of
$L_{\Phi}({\mathcal{M}})$ onto $L_{\Phi}({\mathcal{M}}_{n})$ (see
Theorem 3.4 in \cite{DDP2}) and ${\mathcal{E}}_{n}(x)\geq 0$
whenever $x\geq 0$.

A non-commutative $L_{\Phi}$-martingale with respect to $({\mathcal{M}}_{n})_{n\geq 0}$ is a sequence $x=(x_{n})_{n\geq 0}$
such that $x_{n} \in L_{\Phi}({\mathcal{M}}_{n})$ and
$${\mathcal{E}}_n(x_{n+1})=x_n$$
for any $n \ge 0.$ Let $\|x\|_{\Phi}=\sup_{n\geq 0}\|x_{n}\|_{\Phi}.$ If $\|x\|_{\Phi}
<\infty,$ then $x$ is said to be a bounded $L_{\Phi}$-martingale.

\begin{remark}\label{re:conver}
\begin{enumerate}[\rm (1)]

\item Let $x_{\infty}\in L_{\Phi}({\mathcal{M}}).$ Set
$x_{n}={\mathcal{E}}_{n}(x_{\infty})$ for all $n\geq0.$ Then
$x=(x_{n})$ is a bounded $L_{\Phi}$-martingale and
$\|x\|_{L_{\Phi}({\mathcal{M}})}=\|x_{\infty}\|_{L_{\Phi}({\mathcal{M}})}.$

\item Suppose $\Phi$ is an Orlicz function with $1<p_{\Phi}\le
q_{\Phi}<\infty.$ Then $L_{\Phi}({\mathcal{M}})$ is reflexive. By the standard argument we conclude that any bounded
noncommutative martingale $x=(x_{n})$ in $L_{\Phi}({\mathcal{M}})$
converges to some $x_{\infty}$ in $L_{\Phi}({\mathcal{M}})$ and
$x_{n}={\mathcal{E}}_{n}(x_{\infty})$ for all $n\geq0.$

\item  Let ${\mathcal{M}}$ be a semifinite von Neaumann algebra
with a semifinite normal faithful trace $\tau$. Let
$({\mathcal{M}}_{n})_{n\geq 0}$ be  a filtration of ${\mathcal{M}}$
such that the restriction of $\tau$ to each ${\mathcal{M}}_{n}$ is
still semifinite. Then we can define noncommutative martingales with
respect to $({\mathcal{M}}_{n})_{n\geq 0}$. All results on
noncommutative martingales that will be presented below in this
paper can be extended to this semifinite setting.
\end{enumerate}
\end{remark}

Let $x$ be a noncommutative martingale. The martingale difference sequence of $x,$ denoted
by $dx=(dx_{n})_{n\geq 0},$ is defined as
$$
dx_{0}=x_{0},\quad dx_{n}=x_{n}-x_{n-1},\quad n\geq 1.
$$
Set
$$
S_{C,n}(x)= \Big ( \sum_{k= 0  }^{n}|dx_{k}|^{2} \Big )^{1/2} \quad \mbox{and}
\quad S_{R,n}(x)= \Big ( \sum_{k=0}^{n}|dx_{k}^{*}|^{2} \Big )^{1/2}.
$$
By the preceding discussion, $dx$ belongs to
$L_{\Phi}({\mathcal{M}},\el_{C}^{2})$ (resp.
$L_{\Phi}({\mathcal{M}}, \el_{R}^{2}))$ if and only if $(S_{C,n}(x))_{n\geq 0}$
(resp. $(S_{R,n}(x))_{n\geq 0}$) is a bounded sequence in
$L_{\Phi}({\mathcal{M}});$ in this case,
$$
S_{C}(x)= \Big ( \sum_{k= 0    }^{\infty}|dx_{k}|^{2} \Big )^{1/2} \quad
\mbox{and} \quad S_{R}(x)= \Big ( \sum_{k= 0
}^{\infty}|dx_{k}^{*}|^{2} \Big )^{1/2}
$$
are elements in $L_{\Phi}({\mathcal{M}}).$ These are noncommutative
analogues of the usual square functions in the commutative
martingale theory. It should be pointed out that the two sequences
$S_{C,n}(x)\: \mbox{and}\: S_{R,n}(x)$ may not be bounded in
$L_{\Phi}({\mathcal{M}})$ at the same time.

We define $\mathcal{H}_{C}^{\Phi}({\mathcal{M}})$ (resp. $\mathcal{H}_{R}^{\Phi}({\mathcal{M}})$)
to be the space of all $L_{\Phi}$-martingales with respect to
$({\mathcal{M}}_{n})_{n\geq 0}$ such that $dx \in
L_{\Phi}({\mathcal{M}}, \el_{C}^{2})$ (resp. $dx \in
L_{\Phi}({\mathcal{M}}, \el_{R}^{2})$ ), equipped with the norm
\be
\|x\|_{\mathcal{H}_{C}^{\Phi}({\mathcal{M}})}=\|dx\|_{L_{\Phi}({\mathcal{M}}, \el_{C}^{2})
} \quad \big ( \mbox{resp.} \;
\|x\|_{\mathcal{H}_{R}^{\Phi}({\mathcal{M}})}=\|dx\|_{L_{\Phi}({\mathcal{M}}, \el_{R}^{2}) } \big ).
\ee
$\mathcal{H}_{C}^{\Phi}({\mathcal{M}})$ and $\mathcal{H}_{R}^{\Phi}({\mathcal{M}})$ are
Banach spaces. Note that if $x\in \mathcal{H}_{C}^{\Phi}({\mathcal{M}}),$
\be
\|x\|_{\mathcal{H}_{C}^{\Phi}({\mathcal{M}})}=\sup_{n\geq
0}\|S_{C,n}(x)\|_{L_{\Phi}({\mathcal{M}})}
=\|S_{C}(x)\|_{L_{\Phi}({\mathcal{M }})}.
\ee
The similar equalities hold for $\mathcal{H}_{R}^{\Phi}({\mathcal{M}}).$

Now, we define the Orlicz-Hardy spaces of noncommutative martingales as follows:
If $q_{\Phi} < 2,$
\be
\mathcal{H}_{\Phi}({\mathcal{M}}) = \mathcal{H}_{C}^{\Phi}({\mathcal{M}}) + \mathcal{H}_{R}^{\Phi}({\mathcal{M}}),
\ee
equipped with the norm
\be
\|x\|=\inf \{ \|y\|_{ \mathcal{H}_{C}^{\Phi}({\mathcal{M}})}+\|z\|_{\mathcal{H}_{R}^{\Phi}({\mathcal{M}})}:
x=y+z,\; y \in \mathcal{H}_{C}^{\Phi}({\mathcal{M}}),\; z \in
\mathcal{H}_{R}^{\Phi}({\mathcal{M}})\}.
\ee
If $2\leq p_{\Phi},$
$$
\mathcal{H}_{\Phi}({\mathcal{M}}) = \mathcal{H}_{C}^{\Phi}({\mathcal{M}}) \cap
\mathcal{H}_{R}^{\Phi}({\mathcal{M}}),
$$
equipped with the norm
$$
\|x\|=\max \{ \|x\|_{\mathcal{H}_{C}^{\Phi}({\mathcal{M}})},\;
\|x\|_{\mathcal{H}_{R}^{\Phi}({\mathcal{M}})} \}.
$$

The reason that we have defined $ \mathcal{H}_{\Phi}({\mathcal{M}})$ differently according to $q_{\Phi} < 2 $ or
$2\leq p_{\Phi}$ will become clear in the next section. This has been used first in \cite{PX1,PX2} and also in \cite{LPP}.

\section{An interpolation theorem}\label{inter}

The main result of this section is a Marcinkiewicz type interpolation theorem for noncommutative Orlicz spaces. It is the key to our proof of $\Phi$-moment inequalities of the noncommutative martingales. We first introduce the following definition.

\begin{definition}
Let $\mathcal{N}_1$ (resp. $\mathcal{N}_2$) be
a semifinite von Neumann algebra on a Hilbert space $\mathbb{H}_1$ (resp. $\mathbb{H}_2$)
with a normal semifinite faithful  trace $\nu_1$ (resp. $\nu_2$). A
map $T:L_{0}(\mathcal{N}_1)\rightarrow L_{0}(\mathcal{N}_2)$ is said to
be sublinear if for any  operators $x,y\in L_{0}(\mathcal{N}_1),$
there exist isometrics $u,v\in \mathcal{N}_2$ such that
$$
|T(x+y)|\leq u^{*}|Tx|u+v^{*}|Ty|v,\quad |T(\alpha x)|\leq
|\alpha||Tx|, \;\forall\alpha\in \mathbb{C}.
$$
\end{definition}

This definition of sublinear operators in the noncommutative setting belongs to Q.Xu, which first appeared in Ying Hu's thesis \cite{Hu1} (see also \cite{Hu2}). We recall the definition that a sublinear operator $T: L_{0}(\mathcal{N}_1)\rightarrow L_{0}(\mathcal{N}_2)$ is of {\it weak type} $(p,q)$ with $1 \le p \le q \le \8.$ This means that there is a constant $C>0,$ so that for every $x \in L_p (\N_1)$
\beq\label{eq:WLp} \lambda_{\alpha}(|Tx|) \leq \Big ( \frac{C\|x\|_p}{\alpha} \Big )^q,\quad \forall \alpha > 0.\eeq
If $q = \8,$ it means that $\| T x \|_q \le C \| x \|_p.$

The classical Marcinkiewicz interpolation theorem has been extended to include Orlicz spaces as interpolation classes by A.Zygmund, A.P.Calder\'{o}n, S.Koizumi, I.B.Simonenko, W.Riordan, H.P.Heinig and A.Torchinsky (for references see \cite{M2} and therein). The following result is a noncommutative analogue of the Marcinkiewicz type interpolation theorem for Orlicz spaces.

\begin{theorem}\label{th:Inter}
Let $\mathcal{N}_1$ (resp. $\mathcal{N}_2$) be a semifinite von Neumann algebra on a Hilbert space $\mathbb{H}_1$ (resp. $\mathbb{H}_2$) with a normal semifinite faithful  trace $\nu_1$ (resp. $\nu_2$).
Suppose $1 \le p_{0}<p_{1} \leq \infty.$ Let $T:L_{0}(\mathcal{N}_1) \rightarrow
L_{0}(\mathcal{N}_2)$ be a sublinear operator and simultaneously of weak types $(p_i, p_i)$ for $i=0$ and $i=1.$ If $\Phi$ is an Orlicz function with $p_{0}<p_{\Phi}\le q_{\Phi}<p_{1},$ then there exists a constant $C$ depending only on $p_0, \; p_1$ and $\Phi,$ such that
\begin{equation}\label{strongphi}
\nu_2 (\Phi(|Tx|))\leq C \nu_1 ( \Phi(|x|)),
\end{equation}
for all $x \in L_\Phi(\mathcal{N}_1).$
\end{theorem}

\begin{proof}
At first, we take $p_1<\infty.$ For
$\alpha>0,$ let $x=x_0^\alpha+x_1^\alpha,$ where
$x_0^\alpha=xE_{(\alpha,\infty)}(|x|).$ From the sublinearity of
$T,$ it follows that
\begin{equation}\label{lamb1}
\lambda_{2\alpha} (|Tx|)\leq
\lambda_{\alpha}(|Tx_0^\alpha|)+\lambda_{\alpha}(|Tx_1^\alpha|).
\end{equation}
By \eqref{eq:WLp}, there are two constants $A_0, A_1 >0$ such that for any $\alpha >0$
\begin{equation}\label{weaks0}
\lambda_{\alpha}(|Tx|)\leq
A_0^{p_0}\alpha^{-p_0}\|x\|_{p_0}^{p_0},\quad\forall x\in
L_{p_0}(\mathcal{N}_1),
\end{equation}
\begin{equation}\label{weaks1}
\lambda_{\alpha}(|Tx|)\leq A_1
^{p_1}\alpha^{-p_1}\|x\|_{p_1}^{p_1}, \quad \forall x \in
L_{p_1}(\mathcal{N}_1).
\end{equation}
 Using \eqref{lamb1}, \eqref{weaks0} and \eqref{weaks1} , we have
\beqn
\begin{split}
\nu_2 ( \Phi(|Tx|)) &  =\int_0^\infty
\lambda_{2\alpha}(|Tx|)d\Phi(2\alpha) \\
& \leq \int_0^\infty
\lambda_{\alpha}(|Tx_0^\alpha|)d\Phi(2\alpha) + \int_0^\infty
\lambda_{\alpha}(|Tx_1^\alpha|)d\Phi(2\alpha)\\
& \leq A_0^{p_0}\int_0^\infty
 \alpha^{-p_0} \|x_0^\alpha\|_{p_0}^{p_0}d\Phi(2\alpha)+
A_1 ^{p_1} \int_0^\infty
\alpha^{-p_1}\|x_1^\alpha\|_{p_1}^{p_1}d\Phi(2\alpha)\\
& \leq A_0^{p_0}\int_0^\infty
  \alpha^{-p_0} \nu_1 \big ( |x|^{p_0}E_{(\alpha,\infty)}(|x|) \big ) d \Phi(2\alpha)\\
& \quad + A_1 ^{p_1} \int_0^\infty
  \alpha^{-p_1} \nu_1 \big ( |x|^{p_1}E_{(0,\alpha]}(|x|) \big ) d \Phi(2 \alpha)\\
&\leq A_0^{p_0} \int_0^\infty
  \alpha^{-p_0} \Big ( \int_\alpha^\infty t^{p_0}d \nu_1 (E_{t}(|x|)) \Big ) d \Phi(2\alpha)\\
& \quad + A_1 ^{p_1} \int_0^\infty
  \alpha^{-p_1} \Big ( \int_0^\alpha t^{p_1}d \nu_1 (E_{t}(|x|)) \Big ) d \Phi(2\alpha)\\
& = A_0^{p_0} \int_0^\infty t^{p_0} \Big ( \int_0^t
  \alpha^{-p_0} d \Phi(2\alpha) \Big ) d \nu_1 (E_{t}(|x|))\\
& \quad + A_1^{p_1} \int_0^\infty t^{p_1} \Big ( \int_t^\infty
  \alpha^{-p_1}d\Phi(2\alpha) \Big ) d \nu_1 (E_{t}(|x|)).
\end{split}
\eeqn
By the assumption, we know that $\Phi$ satisfies the $\Delta_2$-condition. This implies that
\be \sup_{t > 0} \frac{t \Phi' (t)}{\Phi(t)} < \8.
\ee
Then, we have
\be \begin{split}
\nu_2 (\Phi(|Tx|)) & \le C_\Phi \Big [ A_0^{p_0} \int_0^\infty t^{p_0} \Big ( \int_0^t
  \alpha^{-p_0-1}\Phi( \alpha)d\alpha \Big ) d \nu_1 (E_{t}(|x|))\\
& \quad + A_1 ^{p_1} \int_0^\infty t^{p_1} \Big ( \int_t^\infty
  \alpha^{-p_1-1}\Phi(\alpha) d \alpha \Big ) d \nu_1 (E_{t}(|x|))\Big ].
\end{split} \ee
On the other hand, by the assumption we have
\be
\int^t_0 s^{- p_0} \Phi (s) \frac{d s}{s} = O (t^{- p_0} \Phi (t))\; \text{and}\; \int^{\8}_t s^{- p_1} \Phi (s) \frac{d s}{s} = O (t^{- p_1} \Phi (t))
\ee
for all $t >0,$ respectively. Hence,
\be
\nu_2 (\Phi(|Tx|)) \le C_\Phi ( A_0^{p_0}+ A_1 ^{p_1} ) \int_0^\infty \Phi(t) d \nu_1 (E_{t}(|x|)) = C \nu_1 (\Phi(|x|)),
\ee
where $C$ depends only on $p_0,\; p_1$ and $\Phi,$ i.e., \eqref{strongphi} holds.

Let $p_1=\infty$ and let $x=x_0^\alpha+x_1^\alpha$ as above. Then
\be \|Tx_1^\alpha\|_{L_\infty}\leq A_1 \|x_1^\alpha\|_{L_\infty}\leq
A_1\alpha \ee and $\lambda_{A_1\alpha}(|Tx_1^\alpha|)=0.$ According to
the above estimate, one obtains
\be
\lambda_{A_1\alpha}(|Tx|)\leq\lambda_{A_1\alpha}(|Tx_0^\alpha|)\leq
\frac{\|Tx_0^\alpha\|^{p_0}_{p_0}}{(A_1\alpha)^{p_0}}
 \le \Big ( \frac{A_0}{A_1} \Big )^{p_0}\alpha^{-p_0}\|x_0^\alpha\|^{p_0}_{p_0}.
\ee
Then, by the same argument as above we have
\beqn
\begin{split}
\nu_2 (\Phi(|Tx|) &  =\int_0^\infty \lambda_{A_1\alpha}(|Tx|)d\Phi(A_1\alpha)\\
& \leq \Big ( \frac{A_0}{A_1} \Big )^{p_0}\int_0^\infty
 \alpha^{-p_0} \|x_0^\alpha\|_{p_0}^{p_0}d\Phi(A_1\alpha)\\
& = \Big ( \frac{A_0}{A_1} \Big )^{p_0}\int_0^\infty
  \alpha^{-p_0} \nu_1 \big ( |x|^{p_0}E_{(\alpha,\infty)}(|x|) \big ) d \Phi(A_1\alpha)\\
& = \Big ( \frac{A_0}{A_1} \Big )^{p_0}\int_0^\infty
  \alpha^{-p_0} \Big ( \int_\alpha^\infty t^{p_0}d \nu_1 (E_{t}(|x|)) \Big ) d \Phi(A_1\alpha)\\
& = \Big ( \frac{A_0}{A_1} \Big )^{p_0}\int_0^\infty t^{p_0} \Big ( \int_0^t
  \alpha^{-p_0}d\Phi(A_1\alpha) \Big ) d \nu_1 (E_{t}(|x|))\\
& \le C_\Phi \Big ( \frac{A_0}{A_1} \Big )^{p_0} \int_0^\infty t^{p_0} \Big ( \int_0^t
  \alpha^{-p_0-1}\Phi( \alpha) d \alpha \Big ) d \nu_1 (E_{t}(|x|))\\
& \le C_\Phi \Big ( \frac{A_0}{A_1} \Big )^{p_0} \int_0^\infty \Phi(t) d \nu_1 (E_{t}(|x|))\\
& = C \nu_1 (\Phi(|x|)),
\end{split}
\eeqn
where $C$ depends only on $p_0,\; p_1$ and $\Phi.$ This completes the proof.
\end{proof}

\begin{remark}\label{re:Inter}
\begin{enumerate}[\rm (1)]

\item If $T$ is of strong type $(p, p),$ i.e., there exists a constant $C>0$ such that $\| T x \|_p \le C \| x \|_p$ for any $x \in L_p (\N),$ then by the Kolmogorov inequality \eqref{eq:kol} we have \be \lambda_{\alpha}(|Tx|)\leq \alpha^{-p} \| Tx \|^p_p \le C^p \alpha^{-p} \|x \|^p_p,\ee
that is, $T$ is of weak type $(p, p).$ Consequently, if $T$ is simultaneously of strong types $(p_i, p_i)$ for $i=0$ and $i=1,$ then the above Theorem still holds.

\item If we only consider the spaces of Hermitian operators, that is,
\be
L_p (\N)_{\mathrm{Her}} = \{x \in L_p (\N):\; x^* =x \},
\ee
the corresponding result of Theorem \ref{th:Inter} also holds. The proof is the same as above and omitted.

\end{enumerate}
\end{remark}

\section{$\Phi$-moment inequalities of martingale transforms}\label{martran}

In the sequel, $ ({\mathcal{M}},\tau)$ always denotes a finite
von Neumann algebra with a normalized normal faithful
trace $\tau$ and $({\mathcal{M}}_{n})_{n\geq 0}$ an increasing filtration
of subalgebras of ${\mathcal{M}} $ which generate ${\mathcal{M}}.$
We keep all notations introduced in the previous sections.

\begin{definition}
Let $\alpha=(\alpha_{n})\subset\mathbb{C}$ be
a sequence. Define a map $T_{\alpha}$  on the family of martingale
difference sequences by  $T_{\alpha}(dx)=(\alpha_{n}dx_{n})$.
$T_{\alpha}$ is called the martingale transform of symbol $\alpha.$
\end{definition}

It is clear that $(\alpha_{n}dx_{n})$ is indeed a martingale
difference sequence. The corresponding martingale is
\be
T_{\alpha}(x)=\sum_n \alpha_{n}dx_{n}.
\ee
The first application of Theorem \ref{th:Inter} is to obtain $\Phi$-moment inequalities of martingale transforms as follows.

\begin{theorem}\label{th:MarTrans} Let $\alpha=(\alpha_{n})\subset\mathbb{C}$ be
a bounded sequence and $T_{\alpha}$ the associated martingale
transform. Let $\Phi$ be an Orlicz function with $1<p_{\Phi} \leq
q_{\Phi}<\infty.$ Then, there is a positive constant $C_{\Phi, \alpha}$
such that for all bounded $L_{\Phi}$-martingales $x=(x_{n})$, we
have
\begin {equation}\label{int-trans}
\tau(\Phi(|T_{\alpha}(x)|)) \leq C_{\Phi, \alpha}\tau(\Phi(| x |)),
\end{equation}
where $C_{\Phi, \alpha}$ depends only on $\Phi$ and $\sup_n |\alpha_n|.$
\end{theorem}

\begin{proof} Let $1<p<p_{\Phi}\leq q_{\Phi}<q<\infty.$ As the consequence of the noncommutative Burkholder-Gundy inequality as proved in Pisier-Xu \cite{PX1} (see Remark 2.4 there) we have
\be
T_{\alpha}:L_{p}({\mathcal{M}})+L_{q}({\mathcal{M}})\rightarrow
L_{p}({\mathcal{M}})+L_{q}({\mathcal{M}})
\ee with $\|T_{\alpha}\|_{p}\leq C_{p, \alpha}$ and $\|T_{\alpha}\|_{q}\leq C_{q, \alpha},$ where $C_{p, \alpha}, C_{q, \alpha}$ are both positive constants depending only on $p, q$ and $\sup_n |\alpha_n|.$ Then, it follows from Theorem \ref{th:Inter} that there is a constant $C_{\Phi, \alpha}$ such that
\be
\tau(\Phi(|T_{\alpha} (x)|))\leq C_{\Phi, \alpha} \tau(\Phi(|x|)),
\ee
as required.
\end{proof}

\begin{remark}
It is proved by Randrianantoanina \cite{R1} that $T_{\alpha}$ is of weak type $(1,1),$ from which we also conclude Theorem \ref{th:MarTrans}.
\end{remark}

\begin{corollary}\label{cor:MarDiff}
Let $\Phi$ be an Orlicz function with $1<p_{\Phi}\leq q_{\Phi}<\infty
.$ Then,
\begin{equation}\label{int-diff}
\tau \Big ( \Phi \Big ( \big |\sum \varepsilon _{n}dx_{n} \big | \Big ) \Big )
\approx \tau \Big ( \Phi \Big ( \big | \sum dx_{n} \big | \Big ) \Big ),\; \forall \varepsilon
_{n}=\pm 1
\end{equation}
for all bounded $L_{\Phi}$-martingales $x=(x_{n}),$ where $``\approx"$ depends only on $\Phi.$
\end{corollary}

Recall that a Banach space $X$ is called a UMD space if for some
$q\in(1,\infty)$ (or equivalently, for every $q \in(1,\infty)$) there
exists a constant $C$ such that for any finite $L_{q}$-martingales
$f$ with values in $X$ one has
\be
\Big \| \sum \varepsilon _{n}df_{n} \Big \|_{L_{q}(\Omega;X)} \leq
C \sup_{n\ge1} \|f_{n} \|_{L_{q}(\Omega;X)} ,\; \forall \varepsilon
_{n}=\pm 1.
\ee
Then, a Banach space $X$ is a UMD space if and only if for any $L_{\8}$-bounded Walsh-Paley martingale $f$ with values in $X$, the series $\sum \varepsilon _{n}df_{n}$ converges in probability (cf., see \cite{Liu1}).

Let $(\Omega,\mathcal {F}, P)$ be a probability space equipped with
$(\mathcal {F}_{n})$ a filtration of $\sigma$-subalgebras  of
$\mathcal {F}$ such that $\cup\mathcal {F}_{n}$ generates $\mathcal
{F}$. Let $(\mathcal{N},\nu)$ be a noncommutative probability
space. Put $\mathcal{M}=L_{\infty}(\Omega,\mathcal {F},
P)\overline{\otimes}\mathcal{N}$ equipped with the tensor product
trace, and $\mathcal{M}_{n}=L_{\infty}(\Omega,\mathcal {F}_{n},
P)\overline{\otimes}\mathcal{N}$ for every $n$. Then $(\mathcal
{M}_{n})$ is a filtration of von Neumann subalgebras  of $\mathcal
{M}$. Recall that
$L_{p}(\mathcal{M})=L_{p}(\Omega;L_{p}(\mathcal{N}))$ for all
$0<p<\infty$. In this case,  the noncommutative $L_{p}$-martingales
with respect to $(\mathcal{M}_{n})$  coincide with the usual
$L_{p}$-martingales with respect to $(\mathcal{F}_{n})$ but with
values in $L_{p}(\mathcal{N})$. Hence, by \eqref{int-diff}, for all
bounded $L_{\Phi}$-martingales $f=(f_{n})$  with values in
$L_{\Phi}(\mathcal{N})$ , we have
\begin {equation}\label{int-umd}
\int_{\Omega} \nu \Big ( \Phi \Big ( \Big | \sum \varepsilon _{n}df_{n} \Big | \Big ) \Big ) d P \approx \int_{\Omega} \nu \big ( \Phi(|f|) \big ) d P ,\; \forall \varepsilon _{n}=\pm 1,
\end{equation}
where $``\approx"$ depends only on $\Phi.$

\begin{corollary}\label{cor:UMD} Let $(\mathcal{N},\nu)$ be a noncommutative probability space and $\Phi$ an Orlicz function with $1<p_{\Phi}\leq q_{\Phi}<\infty.$ Then $L_{\Phi}(\mathcal{N})$ is a UMD space.
\end{corollary}

\begin{proof} Let $f=(f_{n})$ be a $L_{\8}$-bounded Walsh-Paley martingale
with values in $L_{\Phi} ( \mathcal{N}).$ By \eqref{int-umd}, we have
\be
\int_{\Omega} \nu \Big ( \Phi \Big ( \Big | \sum \varepsilon _{n}df_{n} \Big | \Big ) \Big ) d P \leq
C \int_{\Omega} \nu \big ( \Phi(|f|) \big ) d P ,\; \forall \varepsilon _{n}=\pm
1,
\ee
from which it follows that $\Phi(|\sum \varepsilon _{n}df_{n}|)<\infty$  a.e., or $\|\sum \varepsilon _{n}df_{n}\|_{\Phi}<\infty$ a.e.. Therefore, by Remark \ref{re:conver} (2), the series $\sum \varepsilon _{n}df_{n}$
converges almost everywhere. This yields that $L_{\Phi}(\mathcal{N})$ is a UMD space.
\end{proof}

\begin{remark}\label{UMDsemiinfinite}
The above result on the UMD property of  $L_{\Phi}(\mathcal{N})$
remains true when $\nu$ is a normal semifinite faithful trace and $1<p_{\Phi}\leq
q_{\Phi}<\infty.$ Indeed, there exists an increasing family $(e_{j})_{j\in
J}$ of projection of $\mathcal{N}$ such that $\nu(e_{j})<\infty$
for every $j\in J$ and such that $e_{j}$ converges to the unit
element of $\mathcal{N}$ in the strong operator topology. Hence,
$\nu(e_{j}\Phi(|x|))\rightarrow\nu(\Phi(|x|))$  for any
$x\in L_{\Phi}(\mathcal{N})$, since $\Phi(|x|)\in
L_{1}(\mathcal{N})$. Therefore, by approximation, one can easily
reduce the semifinite case to the finite one. Alternately, the
preceding argument continues to work for normal semifinite trace
$\nu$ on $\mathcal{N}$ because the subalgebras
$\mathcal{M}_{n}=L_{\infty}(\Omega,\mathcal {F}_{n},
P)\overline{\otimes}\mathcal{N}$ of
$\mathcal{M}=L_{\infty}(\Omega,\mathcal {F},
P)\overline{\otimes}\mathcal{N}$ satisfy the condition in
Remark \ref{re:conver} (3).
\end{remark}

At the end of this section, by our interpolation result Theorem \ref{th:Inter} we easily obtain the following noncommutative analogue of the Stein inequality for Orlicz spaces.

\begin{theorem}\label{th:Stein} Let $\Phi$ be an Orlicz function with $1<p_{\Phi}\leq
q_{\Phi}<\infty $ and $a=(a_{n})_{n\geq 0}$  a finite sequence in
$L_{\Phi}({\mathcal{M}})$. Then, there exists a constant $C_{\Phi}$ such that
\be
\tau \Big ( \Phi \Big [ \Big ( \sum_{n} |{\mathcal{E}}_{n}(a_{n}) |^{2} \Big )^{\frac{1}{2}} \Big ] \Big ) \leq
C_{\Phi} \tau \Big ( \Phi \Big [ \Big ( \sum_{n}|a_{n}|^{2} \Big )^{\frac{1}{2}} \Big ] \Big ).
\ee
Similar assertion holds for the row subspace
$L_{\Phi}(\mathcal{M}; \el_{R}^{2}).$
\end{theorem}

\begin{proof} Let us consider the von Neumann algebra tensor
product ${\mathcal{M}}\otimes{ \mathcal{B}}(\el^{2})$ with the product
trace $\tau \otimes \mathrm{tr}.$ Evidently, $\tau \otimes \mathrm{tr}$ is a semi-finite
normal faithful trace. Let $L_{\Phi}({\mathcal{M}}\otimes
{\mathcal{B}}(\el^{2}))$ be the associated non-commutative $L_{\Phi}$
space. Then, $L_{\Phi}({\mathcal{M}}\otimes {\mathcal{B}}(\el^{2}))$ is
an interpolation space for the couple $(L_{p}({\mathcal{M}}\otimes
{\mathcal{B}}(\el^{2})),L_{q}({\mathcal{M}}\otimes
{\mathcal{B}}(\el^{2}))),$ where $1<p<p_{\Phi}\leq q_{\Phi}<q<\infty.$
We define
$$
T:L_{p}({\mathcal{M}}\otimes
{\mathcal{B}}(\el^{2}))+L_{q}({\mathcal{M}}\otimes
{\mathcal{B}}(\el^{2}))\rightarrow L_{p}({\mathcal{M}}\otimes
{\mathcal{B}}(\el^{2}))+L_{q}({\mathcal{M}}\otimes
{\mathcal{B}}(\el^{2})),
$$
by
$$T \left(
\begin{matrix} a_{11} & \ldots & a_{1n} & \ldots \\
a_{21} & \ldots & a_{2n} & \ldots \\
\vdots & \vdots & \vdots & \vdots \\
a_{n1} & \ldots & a_{nn} & \ldots \\
\vdots & \vdots & \vdots & \ddots \end{matrix}\right) = \left(
\begin{matrix}
{\mathcal{E}}_{n}(a_{11}) & 0 & 0 & \ldots\\
{\mathcal{E}}_{n}(a_{21}) & 0 & 0 & \ldots \\
\vdots & \vdots & \vdots       & \vdots \\
{\mathcal{E}}_{n}(a_{n1}) & 0 & 0 & \ldots\\
 \vdots & \vdots & \vdots & \ddots
 \end{matrix}\right).
 $$
Theorem 2.3 in \cite{PX1} gives that $T$ is a bounded linear operator on both $L_{p}({\mathcal{M}}\otimes
{\mathcal{B}}(\el^{2}))$ and $L_{q}({\mathcal{M}}\otimes
{\mathcal{B}}(\el^{2})).$ Thus, by Theorem \ref{th:Inter} we obtain the
desired result.
\end{proof}

\begin{remark}\label{re:Stein}
The noncommutative analogue of the classical Stein inequality in $L_p$-spaces is first presented in \cite{PX1}, which is one of key ingredients in their proof of the noncommutative Burkholder-Gundy inequality.
\end{remark}

\section{$\Phi$-moment Khintchine's inequalities}\label{khint}

In this section, we will prove a noncommutative analogue of $\Phi$-moment Khintchine's inequality for Rademacher's random variables.

Let $\mathbb{T}$ be the unit circle of the complex plane equipped
with the normalized Haar measure denoted by $d m.$  Let $\mathcal{M}$
be a finite von Neumann algebra with a normalized normal faithful
trace $\tau.$ Put
$\mathcal{N}=L_{\infty}(\mathbb{T})\overline{\otimes}\mathcal{M}$
equipped with the tensor product trace $\nu=\int\otimes\tau$  and
$\mathcal{A}= \mathcal{H}_{\infty}(\mathbb{T})\overline{\otimes}\mathcal{M}.$
Then, $\mathcal {A}$ is a finite maximal subdiagonal algebras of
$\mathcal {N}$ with respect to $\mathcal{E}=\int\otimes
I_{\mathcal{M}}:\mathcal{N}\rightarrow\mathcal{M}$ (e.g., see \cite{BX}).

\begin{lemma}\label{le:Riesz}  Let $\Phi$ be an Orlicz function with $1<p_{\Phi}\leq
q_{\Phi}<\infty.$ Let $\Phi^{(2)}(t)=\Phi(t^{2}).$ Then, for any
$f \in \mathcal{H}_{\Phi}({\mathcal{N}})$ and $\varepsilon>0,$ there exist two
functions $g,h\in \mathcal{H}_{\Phi^{(2)}}({\mathcal{N}})$ such that
$f=gh$ with
\be
\max \Big \{ \int_{\mathbb{T}} \tau \big [ \Phi (|g|^2) \big ] d m,\; \int_{\mathbb{T}} \tau \big [ \Phi (|h|^2) \big ]d m \Big \} \le \int_{\mathbb{T}} \tau \big ( \Phi(|f|) \big ) d m + \varepsilon.
\ee
\end{lemma}

\begin{proof} Using Theorem 6.2 of \cite{MW}, we obtain
\begin{equation}
\label{intersection}
\mathcal{H}_{p_{\Phi}}({\mathcal{N}}) \cap L_{\Phi}({\mathcal{N}}) = \mathcal{H}_{\Phi}({\mathcal{N}}),\quad
\mathcal{H}_{\Phi}({\mathcal{N}}) \cap L_{\Phi^{(2)}}({\mathcal{N}}) = \mathcal{H}_{\Phi^{(2)}}({\mathcal{N}}).
\end{equation}
Let $w=(f^*f+\varepsilon)^{1/2}$. Then $w \in L_{\Phi}(\mathcal{N})$ and $w^{-1}\in\mathcal{N}.$ Let
$v\in\mathcal{N}$ be a contraction such that $f=vw$. Applying Theorem 4.8 of \cite{BX} to $w^{\frac{1}{2}},$ we have
$w^{\frac{1}{2}}=uh$, where $u$ is a unitary in $\mathcal{N}$ and
$h\in \mathcal{H}^{2 p_{\Phi}}(\mathcal{N})$ such that $h^{-1}\in\mathcal{A}$.
Set $g=vw^{\frac{1}{2}}\,u$. Then $f=gh$, so $g=fh^{-1}$. Since
$g\in \mathcal{H}_{\Phi}({\mathcal{N}})$ and $h^{-1}\in\mathcal{ A}$, $f\in
\mathcal{H}_{\Phi}({\mathcal{N}})$. By \eqref{intersection}, $g,h\in
\mathcal{H}_{\Phi^{(2)}}({\mathcal{N}}).$ The integral estimate is clear.
\end{proof}

\begin{lemma}\label{le:Fourier} Let $\Phi$ be an Orlicz function with $1<p_{\Phi}\leq
q_{\Phi}<\infty.$ Let $\{I_n =(\frac{3^n}{2},3^n]:n\in\mathbb{N}\}$  and $\triangle_{n}$ the Fourier multiplier
by the indicator function $\chi_{I_{n}},$ i.e.
\be
\triangle_{n}(f)(z) = \sum_{k \in I_{n}} \hat{f}(k)z^{k}
\ee
for any trigonometric polynomial $f$ with coefficients in $
L_{\Phi}({\mathcal{M}}).$ Then, there exists a constant $C_{\Phi}>0$
such that
\be
\int_{\mathbb{T}} \tau \Big ( \Phi \Big [ \Big ( \sum_{n}\triangle_{n}(f)^{*}\triangle_{n}(f) \Big )^{\frac{1}{2}} \Big ] \Big )d m \le
C_{\Phi}\int_{\mathbb{T}} \tau \big ( \Phi(|f|) \big )d m,
\ee
for any $ f \in L_{\Phi}({\mathcal{M}}).$
\end{lemma}

\begin{proof} Let
$\mathcal{N}=L_{\infty}(\mathbb{T})\bar{\otimes}\mathcal{M}$
equipped with the tensor product trace $\nu=\int\otimes\tau,\;1<p<\infty,$ then
 $L_{p}(\mathbb{T},L_{p}({\mathcal{M}}))=L_{p}({\mathcal{N}}).$ By Theorem 4 of \cite{Bo} (see also the proof of Theorem {\rm III.1} of \cite{LPP}) there exists a constant $C_{p}>0$ such that for all $f \in L_{p}(\mathbb{T},L_{p}({\mathcal{M}})),$ we have that
\be
\Big \| \Big ( \sum_{n}\triangle_{n}(f)^{*}\triangle_{n}(f) \Big )^{\frac{1}{2}} \Big \|_{L_{p}(\mathbb{T},L_{p}({\mathcal{M}}))} \le C_{p}
\|f\|_{L_{p}(\mathbb{T},L_{p}({\mathcal{M}}))}
\ee
that is,
\be
\Big \| \Big ( \sum_{n}\triangle_{n}(f)^{*}\triangle_{n}(f) \Big )^{\frac{1}{2}} \Big \|_p \le
C_{p} \|f\|_{p},\;\forall \;f\in L_{p}({\mathcal{N}}).
\ee
Since the mapping $T: \N \mapsto \N \bar{\otimes} \mathcal{B} (\el^2)$ is sublinear, where
\be T f = \Big ( \sum_{n}\triangle_{n}(f)^{*}\triangle_{n}(f) \Big )^{\frac{1}{2}},\; \forall f \in \N, \ee
by Theorem \ref{th:Inter} we obtain the required result.
\end{proof}

Let $(\varepsilon_{n})$ be a Rademacher sequence on a probability space $(\Omega, P).$ In the sequel, without specified, $\mathcal{N}$ denotes a semifinite von Neumann algebra on a Hilbert space $\mathbb{H}$ with a normal semifinite faithful trace $\nu.$

\begin{lemma}\label{le:khinp} Let $\Phi$ be an Orlicz function. Suppose $x = (x_0, x_1, \ldots, x_n)$ is a finite sequence in $L_{\Phi} (\mathcal{N}).$
\begin{enumerate}[\rm (1)]

\item If $1< p_{\Phi} \le q_{\Phi} < 2,$ then
\begin{equation}\label{khinp1}
\begin{split}
\int_{\Omega}
\nu \Big ( & \Phi \Big [ \Big | \sum_{k=0}^{n}x_{k}\varepsilon_{k} \Big | \Big ] \Big ) d P \\
& \leq \min C_{\Phi} \Big \{ \nu \Big ( \Phi \Big [ \Big ( \sum_{k=0}^{n}
|x_{k}|^{2} \Big )^{\frac{1}{2}} \Big ] \Big ),\, \nu \Big ( \Phi \Big [ \Big ( \sum_{k=0}^{n}
|x_{k}^{*}|^{2} \Big )^{\frac{1}{2}} \Big ] \Big ) \Big \}.
\end{split}
\end{equation}
Consequently,
\begin{equation*}\label{khin1}
\begin{split}
\int_{\Omega} \nu \Big ( \Phi \Big [ & \Big | \sum_{k=0}^{n}x_{k}\varepsilon_{k} \Big | \Big ] \Big ) d P\\
& \le C_{\Phi} \inf \Big \{ \nu \Big ( \Phi \Big [ \Big ( \sum_{k=0}^{n}
|y_{k}|^{2} \Big )^{\frac{1}{2}} \Big ] \Big ) + \nu \Big ( \Phi \Big [ \Big ( \sum_{k=0}^{n}
|z_{k}^{*}|^{2} \Big )^{\frac{1}{2} } \Big ] \Big ) \Big \},
\end{split}
\end{equation*}
where the infimun runs over all decomposition $x_{k}=y_{k}+z_{k}$
with $y_{k}$ and $z_{k}$ in $L_{\Phi}({\mathcal{N}}).$

\item If  $2 < p_{\Phi}\le q_{\Phi} < \8,$ then
\begin{equation}\label{khinp2}
\begin{split}
\max \Big \{ \nu \Big ( \Phi  & \Big [ \Big ( \sum_{k=0}^{n}
|x_{k}|^{2} \Big )^{\frac{1}{2}} \Big ] \Big ),\, \nu \Big ( \Phi \Big [ \Big ( \sum_{k=0}^{n}
|x_{k}^{*}|^{2} \Big )^{\frac{1}{2}} \Big ] \Big ) \Big \}\\
& \leq \int_{\Omega}
\nu \Big ( \Phi \Big [ \Big | \sum_{k=0}^{n}x_{k}\varepsilon_{k} \Big | \Big ] \Big ) d P.
\end{split}
\end{equation}

\end{enumerate}
\end{lemma}

\begin{proof} (1)\; We define $
T: L_p (\N \otimes {\mathcal{B}}(\el^{2})) \mapsto L_p ( \N \otimes
L_{\8} (\Omega, P))$
by
\be
T (a_{i j}) = \sum_i \varepsilon_i a_{i 1},\quad \forall (a_{ij}) \in L_p (\N \otimes {\mathcal{B}}(\el^{2})).
\ee
Since $L_p (\N, \el^2_C)$ is $1$-complemented in $L_p (\N \otimes {\mathcal{B}}(\el^{2})),$ by the noncommutative Khintchine inequalities \cite{LP1, LPP} we conclude that $T$ is bounded from $L_1 (\N \otimes {\mathcal{B}}(\el^{2}))$ into $L_1 ( \N \otimes L_{\8} (\Omega, P))$ and $L_2 (\N \otimes {\mathcal{B}}(\el^{2}))$ into $L_2 ( \N \otimes L_{\8} (\Omega, P))$ simultaneously. Then, by Theorem \ref{th:Inter} we have
\be
\int_{\Omega}
\nu \Big ( \Phi \Big [ \Big | \sum_{k=0}^{n}x_{k}\varepsilon_{k} \Big | \Big ] \Big ) d P
\leq  C_{\Phi} \nu \Big ( \Phi \Big [ \Big ( \sum_{k=0}^{n} |x_{k}|^{2} \Big )^{\frac{1}{2}} \Big ] \Big ).
\ee
Similarly, if we let
\be
T (a_{i j}) = \sum_j \varepsilon_j a_{1 j},\quad \forall (a_{ij}) \in L_p (\N \otimes {\mathcal{B}}(\el^{2})),
\ee
then we obtain
\be
\int_{\Omega}
\nu \Big ( \Phi \Big [ \Big | \sum_{k=0}^{n}x_{k}\varepsilon_{k} \Big | \Big ] \Big ) d P
\leq  C_{\Phi} \nu \Big ( \Phi \Big [ \Big ( \sum_{k=0}^{n} |x^*_{k}|^{2} \Big )^{\frac{1}{2}} \Big ] \Big ).
\ee
Hence, \eqref{khinp1} holds.

To prove the second inequality take a decomposition $x_k = y_k + z_k.$ Then there exist two isometries $U, V \in \mathcal{N}$ such that
\be \Big | \sum_{k=0}^{n}x_{k}\varepsilon_{k} \Big | \le  U^* \Big | \sum_{k=0}^{n} y_{k}\varepsilon_{k} \Big | U + V^* \Big | \sum_{k=0}^{n} z_{k}\varepsilon_{k} \Big | V. \ee Consequently, by Proposition 4.6 (ii) in \cite{FK} and \eqref{khinp1} we have
\begin{equation*}\label{khin1}
\begin{split}
\int_{\Omega} \nu \Big ( \Phi \Big [ & \Big | \sum_{k=0}^{n}x_{k}\varepsilon_{k} \Big | \Big ] \Big ) d P\\
& \le C_{\Phi} \Big \{ \int_{\Omega} \nu \Big ( \Phi \Big [ \Big | \sum_{k=0}^{n} y_{k}\varepsilon_{k} \Big | \Big ] \Big ) d P + \int_{\Omega} \nu \Big ( \Phi \Big [ \Big | \sum_{k=0}^{n} z_{k}\varepsilon_{k} \Big | \Big ] \Big ) d P \Big \}\\
& \le C_{\Phi} \inf \Big \{ \nu \Big ( \Phi \Big [ \Big ( \sum_{k=0}^{n}
|y_{k}|^{2} \Big )^{\frac{1}{2}} \Big ] \Big ) + \nu \Big ( \Phi \Big [ \Big ( \sum_{k=0}^{n}
|z_{k}^{*}|^{2} \Big )^{\frac{1}{2} } \Big ] \Big ) \Big \},
\end{split}
\end{equation*}
where we have used the fact that $\Phi \in \Delta_2$ in the first inequality.

(2)\; Without loss of generality, we assume that $\{\varepsilon_i\}$ is a (unconditional) basis in $L_p (\Omega, P)$ ($1 < p < \8$), i.e., $\mathrm{span}\{\varepsilon_i\}$ is dense in $L_p (\Omega, P).$ We let
\be
S \Big ( \sum_k x_k \varepsilon_k \Big ) = \left(
\begin{matrix} x_1 & 0 & 0 & \ldots \\
\vdots & \vdots & \vdots & \vdots \\
x_k & 0 & 0 & \ldots \\
\vdots & \vdots & \vdots & \ddots \end{matrix}\right)
\ee
for any finite sequence $\{x_k\}$ in $L_p (\N).$ By the noncommutative Khintchine inequalities \cite{LP1, LPP} we conclude that $S$ is well defined and extends to a bounded operator from $L_p ( \N \otimes L_{\8} (\Omega, P))$ into $L_p (\N \otimes {\mathcal{B}}(\el^{2}))$ for every $2 \le p < \8.$ Hence, by Theorem \ref{th:Inter} we have
\be
\nu \Big ( \Phi \Big [ \Big ( \sum_{k=0}^{n} |x_{k}|^{2} \Big )^{\frac{1}{2}} \Big ] \Big )
 \leq \int_{\Omega}
\nu \Big ( \Phi \Big [ \Big | \sum_{k=0}^{n}x_{k}\varepsilon_{k} \Big | \Big ] \Big ) d P.
\ee
Similarly, if we set
\be
S \Big ( \sum_k x_k \varepsilon_k \Big ) = \left(
\begin{matrix} x_1 & \ldots & x_k & \ldots \\
0 & \ldots & 0 & \ldots \\
\vdots & \vdots & \vdots & \ddots \end{matrix}\right)
\ee
for any finite sequence $\{x_k\}$ in $L_p (\N),$ then we have
\be
\nu \Big ( \Phi \Big [ \Big ( \sum_{k=0}^{n} |x_{k}^{*}|^{2} \Big )^{\frac{1}{2}} \Big ] \Big ) \leq
\int_{\Omega} \nu \Big ( \Phi \Big [ \Big | \sum_{k=0}^{n}x_{k}\varepsilon_{k} \Big | \Big ] \Big ) d P.
\ee
Hence, \eqref{khinp2} holds.
\end{proof}


As following is the noncommutative analogue of $\Phi$-moment version of Khintchine's inequalities for Rademacher's sequences.

\begin{theorem}\label{th:khin} Let $\Phi$ be an Orlicz function and $\{\varepsilon_i\}$ a Rademacher's sequence.
\begin{enumerate}[\rm (1)]

\item If $1<p_{\Phi} \le q_{\Phi}<2,$ then for any finite sequence $\{x_k\}$ in $L_{\Phi}(\N),$
\begin{equation}\label{khin1}
\begin{split}
\int_{\Omega} \nu \Big ( \Phi \Big [ & \Big | \sum_{k=0}^{n}x_{k}\varepsilon_{k} \Big | \Big ] \Big ) d P \\
& \approx \inf \Big \{ \nu \Big ( \Phi \Big [ \Big ( \sum_{k=0}^{n}
|y_{k}|^{2} \Big )^{\frac{1}{2}} \Big ] \Big )+ \nu \Big ( \Phi \Big [ \Big ( \sum_{k=0}^{n}
|z_{k}^{*}|^{2} \Big )^{\frac{1}{2} } \Big ] \Big ) \Big \},
\end{split}
\end{equation}
where the infimun runs over all decomposition $x_{k}=y_{k}+z_{k}$
with $y_{k}$ and $z_{k}$ in $L_{\Phi}({\mathcal{N}})$ and $``
\approx "$ depends only on $\Phi.$

\item If $2<p_{\Phi} \le q_{\Phi}<\infty,$ then for any finite sequence $\{x_k\}$ in $L_{\Phi}(\N),$
\begin{equation}\label{khin2}
\begin{split}
\int_{\Omega} \nu \Big ( \Phi \Big [ & \Big | \sum_{k=0}^{n}x_{k}\varepsilon_{k} \Big | \Big ] \Big ) d P\\
& \approx \max \Big \{ \nu \Big ( \Phi \Big [ \Big ( \sum_{k=0}^{n}
|x_{k}|^{2} \Big )^{\frac{1}{2} } \Big ] \Big ),\, \nu \Big ( \Phi \Big [ \Big ( \sum_{k=0}^{n}
|x_{k}^{*}|^{2} \Big )^{\frac{1}{2} } \Big ] \Big ) \Big \},
\end{split}
\end{equation}
where $`` \approx "$ depends only on $\Phi.$

\end{enumerate}
\end{theorem}

\begin{proof} (1)\; By Lemma \ref{le:khinp} (1), we need only
to prove the lower estimate of \eqref{khin1}. By the Khintchine-Kahane inequality \cite{Pi} and Theorem \ref{th:Inter}, we are reduced to show for any finite sequence $\{x_k \}$ in $L_{\Phi}(\N),$
\begin{equation}\label{khin1'}
\begin{split}
\inf \Big \{ \nu \Big ( & \Phi \Big [ \Big ( \sum_{k=0}^{n} |y_{k}|^{2} \Big )^{ \frac{1}{2}} \Big ] \Big )+
\nu \Big ( \Phi \Big [ \Big ( \sum_{k=0}^{n} |z_{k}^{*}|^{2} \Big )^{\frac{1}{2} } \Big ] \Big ) \Big \}\\
& \leq A' \int_{\mathbb{T}} \nu \Big ( \Phi \Big [ \Big | \sum_{k=0}^{n}x_{k}z^{3^k} \Big | \Big ] \Big ) d m (z),
\end{split}
\end{equation}
where the infimun runs over all decomposition $x_{k}=y_{k}+z_{k}$ with $y_{k}$ and $z_{k}$ in $L_{\Phi}({\mathcal{N}}).$ To prove \eqref{khin1'}, we can clearly assume, by approximation, that
$\nu$ is finite, and even $\nu(1)=1.$ Thus, let $\{x_k \} \subset L_{\Phi}({\mathcal{M}})$ be a
fixed finite sequence, and set
$$
f(z)=\sum_{k=0}^{n}x_{k}z^{3^k},
$$
then $f\in \mathcal{H}_{\Phi}({\mathcal{N}})$.  Given $\varepsilon>0$ let $g$
and $h$ be the two functions in $\mathcal{H}_{\Phi^{(2)}}({\mathcal{N}})$
associated to $f$ and $\varepsilon$ as in Lemma \ref{le:Riesz}. Then for
any $k$
$$
\hat{f}(3^k)=\sum_{0\le m\le 3^k}\hat{g}(3^k-m)\hat{h}(m).
$$
Let
$$
a_k=\sum_{0\le m\le
\frac{3^k}{2}}\hat{g}(3^k-m)\hat{h}(m)\quad\mbox{and}\quad
b_k=\sum_{\frac{3^k}{2}< m\le 3^k}\hat{g}(3^k-m)\hat{h}(m).
$$
Thus we have a decomposition $x_k=a_k+b_k.$ It remains  to estimate
\be \nu \Big ( \Phi \Big [ \Big ( \sum_{k=0}^{n} |a_{k}|^{2} \Big )^{ \frac{1}{2}} \Big ] \Big )\; \text{and}\;
\nu \Big ( \Phi \Big [ \Big ( \sum_{k=0}^{n} |b_{k}^{*}|^{2} \Big )^{ \frac{1}{2}} \Big ] \Big ),\ee
respectively. To this end, let
$$
g_{k}(z)=\sum_{\frac{3^k}{2}< m\le 3^k}\hat{g}(m)z^m.
$$
Observe that
$$
a_k=\widehat{gh}(3^k)=\int_{\mathbb{T}}g_{k}(z)h(z)z^{-3^{k}}dm(z).
$$
Then, by the Jensen and H\"{o}lder inequalities we have
\be
\begin{split}
\nu \Big ( \Phi & \Big [ \Big ( \sum_{k=0}^{n} |a_{k}|^{2} \Big )^{ \frac{1}{2}} \Big ] \Big )\\
& =\nu \Big ( \Phi \Big [ \Big ( \sum_{k=0}^{n}
\Big |\int_{\mathbb{T}}g_{k}(z)h(z)z^{-3^{k}}d m (z) \Big |^{2} \Big )^{ \frac{1}{2}} \Big ] \Big )\\
& \le  \int_{\mathbb{T}} \nu \Big ( \Phi \Big [ \Big ( \sum_{k=0}^{n}
\big |g_{k}(z)h(z)z^{-3^{k}} \big |^{2} \Big )^{ \frac{1}{2}} \Big ] \Big ) d m(z)\\
& = \int_{\mathbb{T}} \nu \Big ( \Phi \Big [ \Big ( h(z)^{*}\sum_{k=0}^{n}
  g_{k}(z)^{*}g_{k}(z)h(z) \Big )^{ \frac{1}{2}} \Big ] \Big ) d m (z)\\
& = \int_{\mathbb{T}} \int_{0}^{\infty} \Phi \Big [ \mu_{t} \Big \{ \Big ( h(z)^{*}\sum_{k=0}^{n}
  g_{k}(z)^{*}g_{k}(z)h(z) \Big )^{ \frac{1}{2}} \Big \} \Big ] d t d m(z)\\
& \le \int_{\mathbb{T}}\int_{0}^{\infty} \Phi \Big [ \mu_{\frac{t}{2}} \big \{ h(z) \big \} \Big ( \mu_{\frac{t}{2}} \Big \{ \sum_{k=0}^{n}
  g_{k}(z)^{*}g_{k}(z) \Big \} \Big )^{ \frac{1}{2}} \Big ] d t d m (z)\\
& \le C_{\Phi} \int_{\mathbb{T}}\int_{0}^{\infty} \Phi \Big [ \mu_{\frac{t}{2}} \big \{ h(z) \big \}^2 + \mu_{\frac{t}{2}} \Big \{ \sum_{k=0}^{n} g_{k}(z)^{*}g_{k}(z) \Big \}  \Big ] d t d m (z)\\
& \le C_{\Phi} \Big \{ \int_{\mathbb{T}} \int_{0}^{\infty} \Phi \big [ \mu_{t} ( |h(z)|^2 ) \big ] d t d m (z) \\
& \quad + \int_{\mathbb{T}} \int_{0}^{\infty} \Phi \Big [ \mu_{t} \Big ( \sum_{k=0}^{n} g_{k}(z)^{*}g_{k}(z) \Big ) \Big ] d t d m (z) \Big \}\\
& \le C_{\Phi} \Big \{ \int_{\mathbb{T}} \nu \big [ \Phi ( |h(z)|^{2} ) \big ] d m + \int_{\mathbb{T}} \nu \Big ( \Phi \Big [ \sum_{k=0}^{n} g_{k}(z)^{*}g_{k}(z) \Big ] \Big ) d m \Big \}.
\end{split}
\ee
Now let $I_{k}=\{m\in \mathbb{Z}\;:\;\frac{3^k}{2}< m\le 3^k\}.$ Then by Lemma \ref{le:Fourier} we have
\be\begin{split}
\int_{\mathbb{T}} \nu \Big ( \Phi \Big [ \sum_{k=0}^{n} g_{k}(z)^{*}g_{k}(z) \Big ] \Big ) d m & = \int_{\mathbb{T}} \nu \Big ( \Phi^{(2)} \Big [ \Big ( \sum_{k=0}^{n} g_{k}(z)^{*}g_{k}(z)\Big )^{1/2} \Big ] \Big ) d m\\
& \le C_{\Phi} \int_{\mathbb{T}} \nu \Big ( \Phi^{(2)} \big [ |g(z)| \big ] \Big )d m\\
& = C_{\Phi} \int_{\mathbb{T}} \nu \Big ( \Phi \big [ |g(z)|^{2}\big ] \Big )d m,
\end{split}\ee
since $1<2 p_{\Phi} \le p_{\Phi^{(2)}} \le q_{\Phi^{(2)}} \le 2 q_{\Phi}<\8.$
Thus, we deduce that
\be
\begin{split}
\nu \Big ( \Phi \Big [ \Big ( \sum_{k=0}^{n} |a_{k}|^{2} \Big )^{ \frac{1}{2}} \Big ] \Big )
\le C_{\Phi} \Big ( \int_{\mathbb{T}} \tau \Big ( \Phi \big [|f| \big ] \Big ) d m + \varepsilon \Big ).
\end{split}
\ee
Similarly, we have
\be
\nu \Big ( \Phi \Big [ \Big ( \sum_{k=0}^{n} |b^*_{k}|^{2} \Big )^{ \frac{1}{2}} \Big ] \Big )
\le C_{\Phi} \Big ( \int_{\mathbb{T}} \tau \Big ( \Phi \big [ |f| \big ] \Big ) d m + \varepsilon \Big ).
\ee
Hence, we obtain \eqref{khin1'} and complete the proof of (1).

(2)\; The upper estimate of \eqref{khin2} immediately follows from
Lemma \ref{le:khinp} (2). To prove the upper estimate of \eqref{khin2}, we consider first the case of that
$x_0,x_1,\cdots,x_n$ are Hermitian operators in
$L_{\Phi} ( \mathcal{N}).$ Define $T$ as in the proof of Lemma \ref{le:khinp} (1). By the noncommutative Khintchine's inequality \cite{LPP} (also see \cite{LP1}) and the fact that $L_p ({\mathcal{N}},\el_{C}^{2})_{\mathrm{Her}}$ is also $1$-complemented in $L_p (\N \otimes \mathcal{B}(\el_2))_{\mathrm{Her}}$ (i.e., $(a_{ij}) \in L_p (\N \otimes \mathcal{B}(\el_2))_{\mathrm{Her}}$ if $a_{ij}$'s are Hermitian operators and $(a_{ij}) \in L_p (\N \otimes \mathcal{B}(\el_2))$), we obtain that $T$ is bounded from $L_p (\N \otimes \mathcal{B}(\el_2))_{\mathrm{Her}}$ into $L_p (\N \otimes L_p (\Omega, P))_{\mathrm{Her}}$ for $2 \le p < \8.$
Consequently, by Theorem \ref{th:Inter} (see Remark \ref{re:Inter} (2) there) there exists a constant $C_{\Phi}$
such that
\be
\int_{\Omega} \nu \Big ( \Phi \Big [ \Big | \sum_{k=0}^{n}x_{k}\varepsilon_{k} \Big | \Big ] \Big ) d P \leq
 C_{\Phi} \nu \Big ( \Phi \Big [ \Big ( \sum_{k=0}^{n} x_{k}^{2} \Big )^{ \frac{1}{2}} \Big ] \Big ).
\ee
The general case follows from the above special case. Indeed, let
$x_{k}=y_{k} + \mathrm{i} z_{k}\;(1\le k\le n),$ where $y_{k},z_{k}$ are
Hermitian operators. Since
\be
y_{k}^{2}+ z_{k}^{2} = \frac{1}{2}\big [ x_{k}^{*}x_{k}+x_{k}x_{k}^{*} \big ],
\ee
we have
\be
\sum_{k=0}^{n} y_{k}^{2} \le\frac{1}{2} \sum_{k=1}^{n} \big [ x_{k}^{*}x_{k}+x_{k}x_{k}^{*} \big ]\; \text{and}\;
\sum_{k=0}^{n}z_{k}^{2} \le \frac{1}{2} \sum_{k=1}^{n} \big [ x_{k}^{*}x_{k}+x_{k}x_{k}^{*} \big ].
\ee
Hence,
\be
\begin{split}
\int_{\Omega} & \nu \Big ( \Phi \Big [ \Big | \sum_{k=0}^{n}x_{k}\varepsilon_{k} \Big | \Big ] \Big ) d P\\
& = \int_{\Omega} \int_{0}^{\infty} \Phi \Big [ \mu_{t} \Big ( \Big | \sum_{k=0}^{n}x_{k}\varepsilon_{k} \Big | \Big ) \Big ] d t d P\\
& \le \int_{\Omega} \int_{0}^{\infty} \Phi \Big [ \mu_{t} \Big ( \Big | \sum_{k=0}^{n}y_{k}\varepsilon_{k} \Big | \Big ) +
\mu_{t} \Big ( \Big | \sum_{k=0}^{n}z_{k}\varepsilon_{k} \Big | \Big ) \Big ] d t d P\\
& \le \frac{1}{2} \int_{\Omega} \int_{0}^{\infty} \Big \{ \Phi \Big [ 2 \mu_{t} \Big ( \Big | \sum_{k=0}^{n}y_{k}\varepsilon_{k} \Big | \Big ) \Big ] + \Phi \Big [ 2
\mu_{t} \Big ( \Big | \sum_{k=0}^{n}z_{k}\varepsilon_{k} \Big | \Big ) \Big ] \Big \} d t d P\\
& \le C_{\Phi} \int_{\Omega} \int_{0}^{\infty} \Big \{ \Phi \Big [ \mu_{t} \Big ( \Big | \sum_{k=0}^{n}y_{k}\varepsilon_{k} \Big | \Big ) \Big ] + \Phi \Big [
\mu_{t} \Big ( \Big | \sum_{k=0}^{n}z_{k}\varepsilon_{k} \Big | \Big ) \Big ] \Big \} d t d P\\
& \le C_{\Phi} \Big \{ \nu \Big ( \Phi \Big [ \Big ( \sum_{k=0}^{n}
y_{k}^{2} \Big )^{ \frac{1}{2}} \Big ] \Big ) + \nu \Big ( \Phi \Big [ \Big ( \sum_{k=0}^{n}
z_{k}^{2} \Big )^{ \frac{1}{2}} \Big ] \Big ) \Big \}\\
& \le C_{\Phi} \nu \Big ( \Phi \Big [ \Big ( \frac{1}{2}\sum_{k=1}^{n}[x_{k}^{*}x_{k}+x_{k}x_{k}^{*} ] \Big )^{ \frac{1}{2}} \Big ] \Big )\\
& = C_{\Phi} \int_{0}^{\infty} \Phi \Big [ \Big ( \mu_{t} \Big \{ \frac{1}{2} \sum_{k=1}^{n}[x_{k}^{*}x_{k}+x_{k}x_{k}^{*}] \Big \} \Big )^{ \frac{1}{2}} \Big ] dt\\
& \le C_{\Phi} \int_{0}^{\infty} \Phi \Big [ \Big ( \mu_{t} \Big \{ \frac{1}{2} \sum_{k=1}^{n}x_{k}^{*}x_{k} \Big \} + \mu_{t} \Big \{ \frac{1}{2} \sum_{k=1}^{n}x_{k}x_{k}^{*} \Big \} \Big )^{1/2} \Big ] d t\\
& \le C_{\Phi} \max \Big \{ \nu \Big ( \Phi \Big [ \Big ( \sum_{k=0}^{n}
|x_{k}|^{2} \Big )^{ \frac{1}{2}} \Big ] \Big ),\, \nu \Big ( \Phi \Big [ \Big ( \sum_{k=0}^{n}
|x_{k}^{*}|^{2} \Big )^{ \frac{1}{2}} \Big ] \Big ) \Big \},
\end{split}\
\ee
where $C_{\Phi}$'s may be different in different lines. This completes the proof.
\end{proof}

\begin{remark}
Note that Khintchine's inequality is valid for $L_1$-norm in both commutative and noncommutative settings (cf., \cite{LPP}).
We could conjecture that the right condition in Theorem \ref{th:khin} (1) should be $q_{\Phi} <2$ without the additional restriction condition $1<p_{\Phi}.$
However, our argument seems to be inefficient in this case. We need new ideas to approach it.
\end{remark}

\section{$\Phi$-moment Burkholder-Gundy's inequalities}\label{BG}

Now, we are in a position to state and prove the $\Phi$-moment version of noncommutative Burkholder-Gundy martingale inequalities.

\begin{theorem}\label{th:BG}
Let $\Phi$ be an Orlicz function and $x=(x_{n})_{n\geq 0}$ a noncommutative $L_{\Phi}$-martingale.
\begin{enumerate}[\rm (1)]

\item If $1<p_{\Phi} \le q_{\Phi}<2,$ then
\begin {equation}\label{eq:BG1}
\begin{split}
\tau \big ( \Phi [ |x|] \big ) \approx \inf \Big \{ \tau \Big ( \Phi \Big [ \Big ( \sum_{k= 0 }^{\infty} |dy_{k}|^{2} \Big )^{ \frac{1}{2}} \Big ] \Big ) + \tau \Big ( \Phi \Big [ \Big ( \sum_{k= 0 }^{\infty} |dz_{k}^{*}|^{2} \Big )^{ \frac{1}{2}} \Big ] \Big ) \Big \}
\end{split}
\end{equation}
where the infimum runs over all decomposition $x_{k}=y_{k}+z_{k}$
with $y_{k}$ in $\mathcal{H}_{C}^{\Phi}({\mathcal{M}})$ and $z_{k}$ in
$\mathcal{H}_{R}^{\Phi}({\mathcal{M}})$ and ``$\approx$" depends only on $\Phi.$

\item If $2 <p_{\Phi} \le q_{\Phi}<\infty,$ then
\begin {equation}\label{eq:BG2}
\tau \big ( \Phi [ |x|] \big ) \approx \max \Big \{ \tau \Big ( \Phi \Big [ \Big ( \sum_{k= 0 }^{\infty} |dx_{k}|^{2} \Big )^{ \frac{1}{2}} \Big ] \Big ),\; \tau \Big ( \Phi \Big [ \Big ( \sum_{k= 0 }^{\infty} |dx_{k}^{*}|^{2} \Big )^{ \frac{1}{2}} \Big ] \Big ) \Big \},
\end{equation}
where ``$\approx$" depends only on $\Phi.$


\end{enumerate}
\end{theorem}

\begin{proof}
(1)\; Let $x$ be any finite martingale in
$L_{\Phi}({\mathcal{M}})$ and $(\varepsilon_{n})$ a Rademacher
sequence on a probability space $(\Omega,P).$ Then, by
\eqref{int-diff} we have
\be
\tau \Big ( \Phi \Big [ \Big | \sum \varepsilon _{n}dx_{n} \Big | \Big ] \Big ) \approx
\tau \Big ( \Phi \Big [ \Big | \sum dx_{n} \Big | \Big ] \Big ).
\ee
Therefore, integrating on $\Omega$ we have
\begin{equation}\label{eqiv}
\tau \Big ( \Phi \Big [ \Big | \sum dx_{n} \Big | \Big ] \Big ) \approx
\int_{\Omega} \tau \Big ( \Phi \Big [ \Big | \sum \varepsilon _{n}dx_{n} \Big | \Big ] \Big ) d P.
\end{equation}
It follows from Theorem \ref{th:khin} (1) that
\be
\tau \Big ( \Phi \Big [ \Big | \sum dx_{n} \Big | \Big ] \Big ) \approx
\inf \Big \{ \tau \Big ( \Phi \Big [ \Big ( \sum_{k=0}^{n} |d y_{k}|^{2} \Big )^{ \frac{1}{2}} \Big ] \Big )+
\tau \Big ( \Phi \Big [  \Big ( \sum_{k=0}^{n} |d z_{k}^{*}|^{2} \Big )^{ \frac{1}{2}} \Big ] \Big ) \Big \},
\ee
where the infimun runs over all decomposition $x_{k}=y_{k}+z_{k}$
with $y_{k}$ and $z_{k}$ in $L_{\Phi}({\mathcal{N}}).$ Then, using
Theorem \ref{th:Stein} we get \eqref{eq:BG1}.

(2)\; Similarly, using \eqref{eqiv} and Theorem \ref{th:khin} (2) we obtain the desired result \eqref{eq:BG2}.
\end{proof}

As follows, we give two examples for illustrating the $\Phi$-moment version of noncommutative Burkholder-Gundy's inequalities obtained above.

\begin{example}\label{ex:pleq}
Let $\Phi (t) = t^a \ln (1 + t^b)$ with $a > 1$ and $b >0.$ It is easy to check that $\Phi$ is an Orlicz function and
\be
p_{\Phi} = a\quad \text{and}\quad q_{\Phi} = a + b.
\ee
When $1< a < a+b<2,$ we have \eqref{eq:BG1}, while $a > 2$ we have \eqref{eq:BG2}. However, when $1<a \le 2 \le a + b$ Theorem \ref{th:BG} gives no information.
\end{example}

\begin{example}\label{ex:p=q=2}
Let $\Phi (t) = t^p (1 + c \sin(p \ln t))$ with $p > 1/(1-2c)$ and $0< c <1/2.$ Then, $\Phi$ is an Orlicz function and
\be
p_{\Phi} = q_{\Phi} = p.
\ee
When $0< c < 1/4,$ $p_{\Phi} = q_{\Phi} =2$ occurs. In this case, Theorem \ref{th:BG} gives no information yet. However, $\Phi$ is equivalent to $t^p$ and so the corresponding Burkholder-Gundy's inequality holds. On the other hand, in general $p_{\Phi} = q_{\Phi} =p$ does not imply that $\Phi$ is equivalent to $t^p$ (see \cite{M1, M2} for details).
\end{example}

\section{Remarks}\label{re}

In this section, we make some remarks on our results and possible further researches.

(1)\; As indicated in Examples \ref{ex:pleq} and \ref{ex:p=q=2}, $\Phi$-moment Burkholder-Gundy's inequalities of noncommutative martingales in the cases of $1< p_{\Phi} \le 2 \le q_{\Phi}< \8$ remain open. Our interpolation argument seems to be inefficient to approach them. (It is clear that our argument is efficient for all the case $1< p_{\Phi} \le q_{\Phi}< \8$ in the commutative setting.) On the other hand, one encounters some substantial
difficulties in trying to adapt the classical techniques, which used stopping times, to the noncommutative setting. As a good substitute for stopping times, Cuculescu's projections \cite{Cu} played an important role for establishing weak-type inequalities \cite{R3, R4} and a noncommutative analogue of the Gundy¡¯s decomposition \cite{PR}. However, these projections do not seem to be powerful enough for noncommutative $\Phi$-moment inequalities (see also \cite{BCPY} for the noncommutative atomic decomposition and \cite{P} for the noncommutative Davis' decomposition). We need new ideas beyond interpolation and Cuculescu's projections.

(2)\; In \cite{BDG}, the authors proved the following $\Phi$-moment martingale inequality: Let $(\Omega, \mathcal{F}, P)$ be a probability space and $\{\mathcal{F}_k\}$ a increasing sequence of $\sigma$-subfields of $\mathcal{F}.$ If $\Phi$ is an Orlicz function satisfying $\Delta_2$-condition, then for any sequence $\{f_k\}$ of nonnegative $\mathcal{F}$-measurable functions
\beq\label{eq:DoobDual}
\mathbb{E} \Phi \Big ( \sum_k \mathbb{E}[f_k | \mathcal{F}_k] \Big ) \le C_{\Phi} \mathbb{E} \Phi \Big ( \sum_k f_k \Big ).
\eeq
(See also \cite{G} for an another proof.) Stopping times and good-$\lambda$ techniques developed by Burkholder {\it etal} \cite{Burk} are two key ingredients in the proof of \eqref{eq:DoobDual}. In $L_p$-cases, \eqref{eq:DoobDual} is the so-called dual version of Doob's maximal inequality. The noncommutative analogue of \eqref{eq:DoobDual} in the $L_p$-case plays a crucial role in Junge's approach \cite{J} to noncommutative Doob's inequality. Unfortunately, our interpolation argument is unavailable in the approach to \eqref{eq:DoobDual} in noncommutative setting for Orlicz functions. As expected, the good-$\lambda$ techniques in the noncommutative setting should be developed and it might be efficient for this goal.

(3)\; We end the paper with a note on $\Phi$-moment inequalities on the conditioned square function $\sigma (f) = \big ( \sum_n \mathbb{E}_{n-1} [|d f_n|^2] \big )^{1/2}$ and maximal function $f^* = \sup_n | f_n|$ for a martingale $f=\{f_n\}.$ Let us recall the $\Phi$-moment version of the classical Burkholder-Davis-Gundy theorem for martingales (see \cite{BDG}): Let $\Phi$ be an Orlicz function satisfying the $\Delta_2$-condition. Then
\beq\label{eq:BDG1}
\mathbb{E}\Phi (f^*) \approx \mathbb{E} \Phi \big [ S(f) \big ]\;\text{with}\; S(f) = \Big ( \sum_n | d f_n |^2 \Big )^{1/2},
\eeq
for all martingales $f,$ where $`` \approx "$ depends only on $\Phi.$ The noncommutative case is surprisingly differen as noted in \cite{JX3}. Indeed, it was shown in \cite{JX3}, Corollary 14, that \eqref{eq:BDG1} does not hold for $\Phi (t) =t$ in general. Instead, a noncommutative analogue of
\beq\label{eq:BDG2}
\mathbb{E} [ S(f) ] \approx \inf \Big \{ \mathbb{E} [ \sigma (g)] + \mathbb{E} \Big ( \sum_n |d h_n| \Big ) \Big \},
 \eeq
 holds as shown in \cite{P}, where the infimum runs over all decompositions $f = g + h$ with $g, h$ being two martingales adapted to the same filtration. Motivated by this result and the commutative case, we would carry out a noncommutative analogue of the $\Phi$-moment version of \eqref{eq:BDG2} elsewhere \cite{BC}. Again, the interpolation argument will play a key role in this problem.


\end{document}